\documentclass[10pt]{amsart}

\usepackage{amsfonts}
\usepackage{amsthm}
\usepackage{amsmath}
\usepackage{graphicx}
\usepackage{amscd,amssymb,amsthm}

\newcounter{minutes}
\divide\time by 60
\newcounter{hours}
\multiply\time by 60 \addtocounter{minutes}{-\time}
\title{On a sum of modified Bessel functions}

\author[\'Arp\'ad Baricz]{\'Arp\'ad Baricz}
\address{Department of Economics, Babe\c{s}-Bolyai University, Cluj-Napoca 400591, Romania} \email{bariczocsi@yahoo.com}

\author{Tibor K. Pog\'any}
\address{Faculty of Maritime Studies, University of Rijeka, Rijeka 51000, Croatia}
\email{poganj@brod.pfri.hr}


\newtheorem{theorem}{Theorem}

\begin{document}

\def\thefootnote{}
\footnotetext{ \texttt{File:~\jobname .tex,
          printed: \number\year-0\number\month-\number\day,
          \thehours.\ifnum\theminutes<10{0}\fi\theminutes}
} \makeatletter\def\thefootnote{\@arabic\c@footnote}\makeatother

\maketitle

\begin{abstract}
In this paper we consider a sum of modified Bessel functions of the first kind of which particular case is used in the study of Kanter's sharp modified Bessel function bound for concentrations of some sums of independent symmetric random vectors. We present some monotonicity and convexity properties for that sum of modified Bessel functions of the first kind, as well as some Tur\'an type inequalities, lower and upper bounds. Moreover, we point out an error in Kanter's paper \cite{kanter} and at the end of the paper we pose an open problem, which may be of interest for further research.
\end{abstract}

\section{Introduction}

Special functions like modified Bessel functions of the first and second kind, $I_{\nu}$ and $K_{\nu},$ are frequently used in probability theory and statistics, see for example the papers of Fotopoulos and Venkata \cite{venkata}, Kanter \cite{kanter}, Marchand and Perron \cite{marchand0,marchand1}, Robert \cite{robert}, Yuan and Kalbfleisch \cite{yuan}, to mention a few. In \cite{kanter} Kanter deduced a sharp bound for the probability that a sum of independent symmetric random vectors lies in a symmetric convex set, by improving in one dimension an inequality first proved by Kolmogorov. For a very recently deduced concentration inequality for sums of independent isotropic random vectors we refer to the paper of Cranston and Molchanov \cite{cranston}. In deducing the above mentioned result in \cite{kanter} the function $\Phi:(0,\infty)\to(0,\infty),$ defined by $\Phi(x)=e^{-x}\left[I_0(x)+I_1(x)\right],$ plays an important role. Some properties of $\Phi$ were deduced in \cite[Lemma 4.4]{kanter} and \cite[Lemma 4.5]{kanter}. Recently, Mattner and Roos \cite{roos}, by using other properties of the function $\Phi$, shortened the proof of \cite[Theorem 4.1]{kanter}, in which Kanter deduced his concentration bound for sums of independent random vectors. The properties proved in \cite[Lemma 1.4]{roos} complement the study of $\Phi$ from \cite{kanter}.

Our aim in this paper is twofold: in one hand to generalize the results on $\Phi$ from \cite{kanter,roos} to the function $\Phi_{\nu}:(0,\infty)\to(0,\infty),$ defined by $$\Phi_{\nu}(x)=e^{-x}x^{-\nu}\left[I_{\nu}(x)+I_{\nu+1}(x)\right],$$ and on the other hand to point out a gap in the proof of \cite[Lemma 4.4]{kanter}, which in turn implies that Kanter's proof of \cite[Theorem 4.1]{kanter} is not complete. To achieve our goal, we present some monotonicity and convexity properties, lower and upper bounds, and Tur\'an type inequalities for the function $\Phi_{\nu}.$ In the study of the Tur\'an type inequalities for the function $\Phi_{\nu}$ one of the key tools is the Neumann integral formula concerning the product of two modified Bessel functions with different parameters. Moreover, at the end of the paper we pose an open problem, which may be of interest for further research.

\section{Monotonicity and convexity properties of the function $\Phi_{\nu}$}
\setcounter{equation}{0}

Before we present the main results of this paper we recall some definitions, which will be used in the sequel. A function
$f:(0,\infty)\rightarrow\mathbb{R}$ is said to be completely monotonic if $f$ has derivatives of all orders and satisfies
   $$(-1)^mf^{(m)}(x)\geq 0$$
for all $x>0$ and $m\in\{0,1,2,\dots\}.$ A function $g \colon (0,\infty)\to(0,\infty)$ is said to be logarithmically convex, or simply log-convex, if its natural logarithm $\ln g$ is convex, that is, for all $x,y>0$ and $\lambda\in[0,1]$ we have
   $$g(\lambda x+(1-\lambda)y) \leq \left[g(x)\right]^{\lambda}\left[g(y)\right]^{1-\lambda}.$$
A similar characterization of log-concave functions also holds. We also note that every completely monotonic function is log-convex, see
\cite[p. 167]{wider}. Now, by definition, a function $h \colon (0,\infty)\rightarrow(0,\infty)$ is said to be geometrically (or multiplicatively)
convex if it is convex with respect to the geometric mean, that is, if for all $x,y>0$ and all $\lambda\in[0,1]$ the inequality
   $$h(x^{\lambda}y^{1-\lambda}) \leq[h(x)]^{\lambda}[h(y)]^{1-\lambda}$$
holds. The function $h$ is called geometrically concave if the above inequality is reversed. Observe that, actually the
geometrical convexity of a function $h$ means that the function $\ln h$ is a convex function of $\ln x$ in
the usual sense. We also note that the differentiable function $g$ is log-convex (log-concave) if and only if
$x \mapsto g'(x)/g(x)$ is increasing (decreasing), while the differentiable function $h$ is geometrically convex (concave) if
and only if the function $x \mapsto xh'(x)/h(x)$ is increasing (decreasing). See for example \cite{geom} for more details on geometrically convex (concave) functions and their relations with continuous univariate distributions.

The following result is motivated by \cite[Lemma 4.5]{kanter} and \cite[Lemma 1.4]{roos}. Part {\bf a} of Theorem 1 generalizes the statement that $\Phi$ is completely monotonic, see \cite[Lemma 1.4]{roos}, while part {\bf d} provides a generalization of \cite[Lemma 4.5]{kanter}. The right-hand side of \eqref{bound} extends inequality \cite[eq. (9)]{roos}.

\begin{theorem}
The following assertions are valid:
\begin{enumerate}
\item[\bf a.] $\Phi_{\nu}$ is completely monotonic on $(0,\infty)$ for all $\nu\geq -\frac{1}{2};$
\item[\bf b.] $\Phi_{\nu}$ is log-convex and geometrically concave on $(0,\infty)$ for all $\nu\geq -\frac{1}{2};$
\item[\bf c.] $\nu\mapsto \Phi_{\nu}(x)$ is decreasing on $[0,\infty)$ for all $x>0;$
\item[\bf d.] $x\mapsto x^{\nu+\frac{1}{2}}\Phi_{\nu}(x)$ is increasing on $(0,\infty)$ for all $\nu\geq-\frac{1}{2};$
\item[\bf e.] $x\mapsto \left(x+\frac{\nu}{2}\right)^{\nu+\frac{1}{2}}\Phi_{\nu}(x)$ is increasing on $(0,\infty)$ for all $\nu\geq-\frac{1}{2};$
\item[\bf f.] $x\mapsto \left(x+\frac{\nu}{2}+\frac{1}{4}\right)^{\nu+\frac{1}{2}}\Phi_{\nu}(x)$ is increasing on $(0,\infty)$ for all $\nu\geq-\frac{1}{2};$
\item[\bf g.] the inequality
\begin{equation}\label{bound}
\frac{\left(\nu+\frac{1}{2}\right)^{\nu+\frac{1}{2}}}
{2^{2\nu+\frac{1}{2}}\Gamma(\nu+1)\left(x+\frac{\nu}{2}+\frac{1}{4}\right)^{\nu+\frac{1}{2}}}<
\Phi_{\nu}(x)<\sqrt{\frac{2}{\pi}}\cdot\frac{1}
{\left(x+\frac{\nu}{2}+\frac{1}{4}\right)^{\nu+\frac{1}{2}}}\end{equation}
is valid for all $\nu>-\frac{1}{2}$ and $x>0;$
\end{enumerate}
\end{theorem}

\begin{proof}
Observe that since
$$I_{-\frac{1}{2}}(x)=\sqrt{\frac{2}{\pi x}}\cosh x\ \ \ \ \mbox{and}\ \ \ \ I_{\frac{1}{2}}(x)=\sqrt{\frac{2}{\pi x}}\sinh x,$$
we have that $\Phi_{-\frac{1}{2}}(x)=\sqrt{\frac{2}{\pi}}.$ Thus, in what follows we assume that $\nu>-\frac{1}{2}.$

{\bf a.} We proceed somewhat similar as in the case of complete monotonicity of the function $x\mapsto e^{-x}x^{-\nu}I_{\nu}(x).$ About this function we know that it is completely monotonic on $(0,\infty)$ for all $\nu\geq -\frac{1}{2}.$ The case $\nu>-\frac{1}{2}$ was proved by N{\aa}sell \cite{nasell} and the case $\nu=-\frac{1}{2}$ was pointed out in \cite{bariczpams}. In \cite{nasell} N{\aa}sell used the integral representation \cite[p. 79]{watsongn}
\begin{equation}\label{mod}
I_{\nu}(x)=\frac{\left(\frac{1}{2}x\right)^{\nu}}{\sqrt{\pi}\Gamma\left(\nu+\frac{1}{2}\right)}
\int_{-1}^1\left(1-t^2\right)^{\nu-\frac{1}{2}}e^{-xt}dt,\ \ \ \nu>-\frac{1}{2},\end{equation}
to prove the above mentioned complete monotonicity. Now, if we change $\nu$ to $\nu+1$ in \eqref{mod} and we use integration by parts, we obtain
\begin{align*}
I_{\nu+1}(x)&=\frac{\left(\frac{1}{2}x\right)^{\nu+1}}{\sqrt{\pi}\Gamma\left(\nu+\frac{3}{2}\right)}
\int_{-1}^1\left(1-t^2\right)^{\nu+\frac{1}{2}}e^{-xt}dt\\
&=\frac{\frac{1}{2}\left(\frac{1}{2}x\right)^{\nu}}{\sqrt{\pi}\Gamma\left(\nu+\frac{3}{2}\right)}
\int_{-1}^1\left(1-t^2\right)^{\nu+\frac{1}{2}}xe^{-xt}dt\\
&=-\frac{\left(\frac{1}{2}x\right)^{\nu}}{\sqrt{\pi}\Gamma\left(\nu+\frac{1}{2}\right)}
\int_{-1}^1\left(1-t^2\right)^{\nu-\frac{1}{2}}te^{-xt}dt,\end{align*}
which in view of \eqref{mod} yields that
\begin{equation}\label{modsum}
\Psi_{\nu}(x)=\sqrt{\pi}2^{\nu}\Gamma\left(\nu+\frac{1}{2}\right)\Phi_{\nu}(x)=\int_{-1}^1(1-t)\left(1-t^2\right)^{\nu-\frac{1}{2}}e^{-x(1+t)}dt.
\end{equation}
Consequently, we obtain
$$(-1)^n\left[\Psi_{\nu}(x)\right]^{(n)}=\int_{-1}^1(1-t)(1+t)^n\left(1-t^2\right)^{\nu-\frac{1}{2}}e^{-x(1+t)}dt>0$$
for all $\nu>-\frac{1}{2},$ $x>0$ and $n\in\{0,1,\dots\}.$ Thus, the function $\Psi_{\nu}$ is completely monotonic on $(0,\infty)$ for all $\nu>-\frac{1}{2},$ as well as the function $\Phi_{\nu}.$

{\bf b.} The first part of the assertion follows from part {\bf a}. Namely, it is known that every completely monotonic function is log-convex, see \cite[p. 167]{wider}.

Now, for convenience let us introduce the notation $$\phi_{\nu}(x)=x^{-\nu}\left[I_{\nu}(x)+I_{\nu+1}(x)\right].$$ Then by using the recurrence relations \cite[p. 79]{watsongn}
$$\left[x^{-\nu}I_{\nu}(x)\right]'=x^{-\nu}I_{\nu+1}(x),\ \ \ \ I_{\nu+2}(x)-I_{\nu}(x)=-\frac{2\nu+1}{x}I_{\nu+1}(x)$$
we have
$$\phi_{\nu}'(x)-\phi_{\nu}(x)=-(2\nu+1)x^{-(\nu+1)}I_{\nu+1}(x)$$
and consequently
$$\frac{\Phi_{\nu}'(x)}{\Phi_{\nu}(x)}=\frac{\phi_{\nu}'(x)-\phi_{\nu}(x)}{\phi_{\nu}(x)}=-\frac{2\nu+1}{x}\frac{I_{\nu+1}(x)}{I_{\nu}(x)+I_{\nu+1}(x)},$$
that is,
\begin{equation}\label{logder}
\frac{x\Phi_{\nu}'(x)}{\Phi_{\nu}(x)}=-(2\nu+1)\frac{1}{\frac{I_{\nu}(x)}{I_{\nu+1}(x)}+1}.
\end{equation}
Since the function $x\mapsto I_{\nu+1}(x)/I_{\nu}(x)$ is increasing on $(0,\infty)$ for all $\nu\geq -\frac{1}{2}$ (see \cite{swatson2} or \cite[p. 446]{yuan}), it follows that $x\mapsto 1/\left[{I_{\nu}(x)}/{I_{\nu+1}(x)}+1\right]$ is increasing on $(0,\infty)$ for all $\nu\geq -\frac{1}{2}.$ Now, by using \eqref{logder} clearly the function
$x\mapsto {x\Phi_{\nu}'(x)}/{\Phi_{\nu}(x)}$ is decreasing on $(0,\infty)$ for all $\nu\geq -\frac{1}{2},$ as we required.

{\bf c.} By using the infinite series representation of the modified Bessel function of the first kind \cite[p. 77]{watsongn}
$$I_{\nu}(x)=\sum_{n\geq0}\frac{\left(\frac{1}{2}x\right)^{2n+\nu}}{n!\Gamma(n+\nu+1)}$$ we have
\begin{equation}\label{power}\phi_{\nu}(2x)=\sum_{n\geq 0}a_n(\nu)x^{2n}+\sum_{n\geq 0}b_n(\nu)x^{2n+1},\end{equation}
where $$a_n(\nu)=\frac{1}{2^{\nu}n!\Gamma(n+\nu+1)}\ \ \ \ \mbox{and} \ \ \ \ b_n(\nu)=\frac{1}{2^{\nu}n!\Gamma(n+\nu+2)}.$$
Thus, to prove that the function $\nu\mapsto \Phi_{\nu}(x)=e^{-x}\phi_{\nu}(x)$ is decreasing on $[0,\infty)$ for all $x>0,$ it is enough to show that $\nu\mapsto a_n(\nu)$ and $\nu\mapsto b_n(\nu)$ are decreasing on $[0,\infty)$ for all $n\in\{0,1,\dots\},$ that is, for all $\nu\geq0$ and $n\in\{0,1,\dots\}$ we have
\begin{equation}\label{annu}\frac{\partial\log(a_n(\nu))}{\partial\nu}=\frac{1}{a_n(\nu)}\frac{\partial a_n(\nu)}{\partial\nu}=-\log 2-\psi(n+\nu+1)<0,\end{equation}
\begin{equation}\label{bnnu}\frac{\partial\log(b_n(\nu))}{\partial\nu}=\frac{1}{b_n(\nu)}\frac{\partial b_n(\nu)}{\partial\nu}=-\log 2-\psi(n+\nu+2)<0,\end{equation}
where $\psi(x)=\Gamma'(x)/\Gamma(x)$ denotes the digamma function. Since $\psi(x)>0$ for all $x>x^*,$ where $x^*\simeq 1.461632144\dots$ is the abscissa of the minimum of the $\Gamma$ function, the inequality \eqref{annu} clearly holds for all $\nu\geq0$ and $n\in\{1,2,\dots\},$ while the inequality \eqref{bnnu} clearly holds for all $\nu\geq0$ and $n\in\{0,1,\dots\}.$ Thus, we just need to verify the inequality \eqref{annu} when $n=0,$ that is, $\psi(\nu+1)+\log2>0$ for $\nu\geq0.$ When $\nu=0$ this is true, since $\psi(1)\simeq-0.5772156649\dots$ and $\log2=0.6931471805{\dots}.$ Now, suppose that $\nu>0.$ According to Batir \cite[Lemma 1.7]{batir} the inequality $\psi(\nu+1)>\log\left(\nu+\frac{1}{2}\right)$ is valid for all $\nu>0.$ Thus $\psi(\nu+1)+\log2>\log(2\nu+1)>0$ for all $\nu>0,$ as we required.

{\bf d.} \& {\bf e.} \& {\bf f.} We shall use \eqref{logder} and the well-known inequalities of Soni \cite{soni} and N{\aa}sell \cite{nasell2}. Namely, by using \eqref{logder} and Soni's inequality $I_{\nu}(x)>I_{\nu+1}(x),$ which holds for all $\nu\geq -\frac{1}{2}$ and $x>0,$ we obtain for $\nu\geq -\frac{1}{2}$ and $x>0$ the inequality
$$\left[\log\Phi_{\nu}(x)\right]'=\frac{\Phi_{\nu}'(x)}{\Phi_{\nu}(x)}>-\left(\nu+\frac{1}{2}\right)\frac{1}{x}=-\left(\log x^{\nu+\frac{1}{2}}\right)'.$$
Similarly, by using N{\aa}sell's inequality \begin{equation}\label{nasel}I_{\nu}(x)>\left(1+\frac{\nu}{x}\right)I_{\nu+1}(x),\end{equation} which holds for all $\nu>-1$ and $x>0,$ we obtain that
$$\left[\log\Phi_{\nu}(x)\right]'=\frac{\Phi_{\nu}'(x)}{\Phi_{\nu}(x)}>-\left(\nu+\frac{1}{2}\right)\frac{1}{x+\frac{\nu}{2}}=-\left[\log \left(x+\frac{\nu}{2}\right)^{\nu+\frac{1}{2}}\right]',$$
where $\nu\geq -\frac{1}{2}$ and $x>0.$ These inequalities imply that the functions $x\mapsto \log\left[x^{\nu+\frac{1}{2}}\Phi_{\nu}(x)\right]$ and $x\mapsto \log\left[\left(x+\frac{\nu}{2}\right)^{\nu+\frac{1}{2}}\Phi_{\nu}(x)\right]$ are increasing on $(0,\infty)$ for all $\nu\geq-\frac{1}{2}.$

Now, for $x>0$ and $\nu\geq0$ let us consider the inequality \cite[eq. (20)]{segura}
$$\frac{I_{\nu}(x)}{xI_{\nu-1}(x)}<\frac{1}{\nu-\frac{1}{2}+\sqrt{x^2+\left(\nu-\frac{1}{2}\right)^2}},$$
which implies that
$$\frac{I_{\nu}(x)}{I_{\nu+1}(x)}>\frac{\sqrt{x^2+\left(\nu+\frac{1}{2}\right)^2}}{x}+\frac{\nu+\frac{1}{2}}{x}>1+\frac{\nu+\frac{1}{2}}{x},$$
where $x>0$ and $\nu\geq-1.$ Observe that this improves N{\aa}sell's inequality \eqref{nasel} and implies the inequality
$$\left[\log\Phi_{\nu}(x)\right]'=\frac{\Phi_{\nu}'(x)}{\Phi_{\nu}(x)}>-\left(\nu+\frac{1}{2}\right)\frac{1}{x+\frac{\nu}{2}+\frac{1}{4}}=-\left[\log \left(x+\frac{\nu}{2}+\frac{1}{4}\right)^{\nu+\frac{1}{2}}\right]',$$
where $\nu\geq -\frac{1}{2}$ and $x>0.$ Thus, the function $x\mapsto \log\left[\left(x+\frac{\nu}{2}+\frac{1}{4}\right)^{\nu+\frac{1}{2}}\Phi_{\nu}(x)\right]$ is increasing on $(0,\infty)$ for all $\nu\geq-\frac{1}{2},$ as we required.

{\bf g.} It is known that $e^{-x}>1-x$ for $x>0$ and thus
$$\left(\frac{2-t}{t}\right)^{\nu+\frac{1}{2}}<\left(\frac{2}{t}\right)^{\nu+\frac{1}{2}}e^{-\frac{t}{2}\left(\nu+\frac{1}{2}\right)}$$
for all $\nu>-\frac{1}{2}$ and $t\in(0,2].$ By using this inequality together with \eqref{modsum} we obtain
\begin{align*}\Phi_{\nu}(x)&=\frac{1}{\sqrt{\pi}2^{\nu}\Gamma\left(\nu+\frac{1}{2}\right)}\int_0^2\left(\frac{2-t}{t}\right)^{\nu+\frac{1}{2}}t^{2\nu}e^{-xt}dt\\
&<\frac{1}{\sqrt{\pi}2^{\nu}\Gamma\left(\nu+\frac{1}{2}\right)}\int_0^2\left(\frac{2}{t}\right)^{\nu+\frac{1}{2}}e^{-\frac{t}{2}\left(\nu+\frac{1}{2}\right)}t^{2\nu}e^{-xt}dt\\
&=\sqrt{\frac{2}{\pi}}\frac{1}{\Gamma\left(\nu+\frac{1}{2}\right)}\int_0^2t^{\nu-\frac{1}{2}}e^{-\left(x+\frac{\nu}{2}+\frac{1}{4}\right)t}dt\\
&<\sqrt{\frac{2}{\pi}}\frac{1}{\Gamma\left(\nu+\frac{1}{2}\right)}\int_0^{\infty}t^{\nu-\frac{1}{2}}e^{-\left(x+\frac{\nu}{2}+\frac{1}{4}\right)t}dt\\
&=\sqrt{\frac{2}{\pi}}\cdot\frac{1}
{\left(x+\frac{\nu}{2}+\frac{1}{4}\right)^{\nu+\frac{1}{2}}},\end{align*}
where $\nu>-\frac{1}{2}$ and $x>0.$ Here in the last step we used the well-known formula
$$\int_0^{\infty}t^{\alpha}e^{-\beta t}dt=\frac{\Gamma(\alpha+1)}{\beta^{\alpha+1}},$$
which is valid for all $\alpha>-1$ and $\beta>0.$

Now, for the left-hand side of \eqref{bound} we use part {\bf f} together with \eqref{power} for $x=0.$ Namely, for all $x>0$ and $\nu>-\frac{1}{2}$ we have
\begin{align*}\left(x+\frac{\nu}{2}+\frac{1}{4}\right)^{\nu+\frac{1}{2}}\Phi_{\nu}(x)&>\frac{1}{2^{\nu+\frac{1}{2}}}\left(\nu+\frac{1}{2}\right)^{\nu+\frac{1}{2}}\Phi_\nu(0)\\&=
\frac{1}{2^{\nu+\frac{1}{2}}}\left(\nu+\frac{1}{2}\right)^{\nu+\frac{1}{2}}\frac{1}{2^{\nu}\Gamma(\nu+1)}.\end{align*}
\end{proof}

The next result is motivated by the Tur\'an type inequalities for modified Bessel functions of the first and second kind. For more details see the recent paper \cite{baricz} and the references therein.

\begin{theorem}
The function $\nu\mapsto \Phi_{\nu}(x)$ is log-concave on $(-1,\infty)$ for $x>0,$ while the function $\nu\mapsto \Psi_{\nu}(x)$ is completely monotonic and log-convex on $\left(-\frac{1}{2},\infty\right)$ for $x>0.$ Moreover, the next Tur\'an type inequality is valid for all $\nu>\frac{1}{2}$ and $x>0$
\begin{equation}\label{turan}0<\left[\Phi_{\nu}(x)\right]^2-\Phi_{\nu-1}(x)\Phi_{\nu+1}(x)\leq\frac{1}{\nu+\frac{1}{2}}\left[\Phi_{\nu}(x)\right]^2.\end{equation}
In addition, the left-hand side of \eqref{turan} holds true for all $\nu>-1$ and $x>0.$
\end{theorem}

\begin{proof} First we show that for each fixed $b\in(0,2]$ and each $x>0,$ the
function $\nu\mapsto \Phi_{\nu+b}(x)/\Phi_{\nu}(x)$ is decreasing, where
$\nu\geq-(b+1)/2,$ $\nu>-1.$ For this, we consider Neumann's formula \cite[p. 441]{watsongn}
\begin{equation}\label{neumann}
I_{\alpha}(x)I_{\beta}(x)=\frac{2}{\pi}\int_0^{\frac{\pi}{2}}I_{\alpha+\beta}(2x\cos\theta)\cos\left((\alpha-\beta)\theta\right)d\theta,
\end{equation}
which holds for all $\alpha+\beta>-1.$ Using this we find that for
$2\nu+\varepsilon+b>-1$
\begin{align*}\alpha_{\nu}(x)&=I_{\nu}(x)I_{\nu+b+\varepsilon}(x)-I_{\nu+b}(x)I_{\nu+\varepsilon}(x)\\&=-\frac{4}{\pi}\int_{0}^{\frac{\pi}{2}}
I_{2\nu+b+\varepsilon}(2x\cos\theta)\sin(b\theta)\sin(\varepsilon\theta)d\theta,\end{align*}
\begin{align*}\beta_{\nu}(x)&=I_{\nu}(x)I_{\nu+b+\varepsilon+1}(x)-I_{\nu+b}(x)I_{\nu+\varepsilon+1}(x)\\&\ \ \ \ +
I_{\nu+1}(x)I_{\nu+b+\varepsilon}(x)-I_{\nu+b+1}(x)I_{\nu+\varepsilon}(x)\\
&=-\frac{8}{\pi}\int_{0}^{\frac{\pi}{2}}
I_{2\nu+b+\varepsilon+1}(2x\cos\theta)\sin(b\theta)\sin(\varepsilon\theta)\cos\theta d\theta,\end{align*}
which are negative for all $b\in(0,2]$ and $\varepsilon\in(0,2].$
Consequently, the expression
$$\Phi_{\nu}(x)\Phi_{\nu+b+\varepsilon}(x)-\Phi_{\nu+b}(x)\Phi_{\nu+\varepsilon}(x)=e^{-2x}x^{-(2\nu+b+\varepsilon)}
\left[\alpha_{\nu}(x)+\beta_{\nu}(x)+\alpha_{\nu+1}(x)\right]$$
is negative, that is, we obtain
$\Phi_{\nu+b+\varepsilon}(x)/\Phi_{\nu+\varepsilon}(x)<\Phi_{\nu+b}(x)/\Phi_{\nu}(x).$ Now, since $\nu\mapsto \Phi_{\nu+b}(x)/\Phi_{\nu}(x)$ is
decreasing, it follows that the function $\nu\mapsto
\log[\Phi_{\nu+b}(x)]-\log[\Phi_{\nu}(x)]$ is decreasing too. This implies
that the function $\nu\mapsto \partial\log[\Phi_{\nu}(x)]/\partial\nu$ is decreasing on
$(-1,\infty).$ Now, since $\nu\mapsto\Phi_{\nu}(x)$ is log-concave, it follows that for all $\nu_1,\nu_2>-1,$ $x>0$ and $\alpha\in[0,1]$ we have
$$\Phi_{\alpha\nu_1+(1-\alpha)\nu_2}(x)\geq \left[\Phi_{\nu_1}(x)\right]^{\alpha}\left[\Phi_{\nu_2}(x)\right]^{1-\alpha},$$
and choosing $\nu_1=\nu-1,$ $\nu_2=\nu+1$ and $\alpha=\frac{1}{2}$ we arrive at left-hand side of the Tur\'an type inequality \eqref{turan}, but just for $\nu>0.$ In what follows we show that this inequality is actually valid for all $\nu>-1.$ For this, observe that
$${}_{\Phi}\Delta_{\nu}(x)=\left[\Phi_{\nu}(x)\right]^2-\Phi_{\nu-1}(x)\Phi_{\nu+1}(x)=
e^{-2x}x^{-2\nu}\left[\Delta_{\nu}(x)+\Delta_{\nu+1}(x)+\Theta_{\nu}(x)\right],$$
where
$$\Delta_{\nu}(x)=\left[I_{\nu}(x)\right]^2-I_{\nu-1}(x)I_{\nu+1}(x)$$
and
$$\Theta_{\nu}(x)=I_{\nu}(x)I_{\nu+1}(x)-I_{\nu-1}(x)I_{\nu+2}(x).$$
Since $\Delta_{\nu}(x)>0$ for all $x>0$ and $\nu>-1$ (see for example \cite{bariczpams}), to prove the left-hand side of \eqref{turan}, we just need to show that the expression $\Theta_{\nu}(x)$ is positive. By using \eqref{neumann} we obtain
\begin{align*}\Theta_{\nu}(x)&=\frac{2}{\pi}\int_0^{\frac{\pi}{2}}I_{2\nu+1}(2x\cos\theta)\left(\cos\theta-\cos(3\theta)\right)d\theta\\
&=\frac{8}{\pi}\int_0^{\frac{\pi}{2}}I_{2\nu+1}(2x\cos\theta)(\cos\theta)\left(\sin^2\theta\right)d\theta>0\end{align*}
for all $\nu>-1$ and $x>0.$

Finally, since
$$(-1)^m\frac{\partial^m\Psi_{\nu}(x)}{\partial\nu^m}=\int_{-1}^1(1-t)(1-t^2)^{\nu-\frac{1}{2}}\left(\log\frac{1}{1-t^2}\right)^me^{-x(1+t)}dt>0$$
for all $\nu>-\frac{1}{2},$ $x>0$ and $m\in\{0,1,\dots\},$ indeed the function $\nu\mapsto \Psi_{\nu}(x)$ is completely monotonic and log-convex on $\left(-\frac{1}{2},\infty\right)$ for $x>0.$ Consequently, for all $\nu_1,\nu_2>-\frac{1}{2},$ $x>0$ and $\alpha\in[0,1]$ we have
$$\Psi_{\alpha\nu_1+(1-\alpha)\nu_2}(x)\leq \left[\Psi_{\nu_1}(x)\right]^{\alpha}\left[\Psi_{\nu_2}(x)\right]^{1-\alpha},$$
and choosing $\nu_1=\nu-1,$ $\nu_2=\nu+1$ and $\alpha=\frac{1}{2}$ we arrive at the Tur\'an type inequality
$$\left[\Psi_{\nu}(x)\right]^2-\Psi_{\nu-1}(x)\Psi_{\nu+1}(x)\leq0,$$
which is equivalent to the right-hand side of \eqref{turan}.
\end{proof}

\section{Remarks on Kanter's paper}
\setcounter{equation}{0}

In this section we would like to point out an error in Kanter's paper \cite{kanter}. Namely, in the beginning of the proof of \cite[Lemma 4.4]{kanter} the author claimed that the inequality
\begin{equation}\label{conjinteg}\frac{1}{\pi}\int_0^{\pi}e^{-2r(1-\cos t)}(1+\cos t)dt\geq \frac{1}{\pi}\int_0^{\pi}(\cos t)^{2r}(1+\cos t)dt,\end{equation}
that is,
\begin{equation}\label{conj}\Phi(2r)\geq p(2r)=C_{2r}^r2^{-2r},\end{equation}
is valid for $r$ nonnegative integer. However, the proof of the inequality \cite[p. 232]{kanter}
\begin{equation}\label{mono}\frac{\Phi(2r)}{p(2r)}\leq \frac{\Phi(2r+2)}{p(2r+2)}\end{equation}
is not correct, and hence the proof of \eqref{conj} is not complete. Namely, there is a typographical error in the definition of $p(2r)$ in \cite[p. 232]{kanter}, and in the proof of the above inequality the author used the inequalities
$$\frac{\Phi(2r)}{p(2r)}\leq \frac{2r+1}{2r+2}\sqrt{\frac{r+1}{r}}\cdot\frac{\Phi(2r+2)}{p(2r+2)}\leq\frac{\Phi(2r+2)}{p(2r+2)},$$
but it is easy to see that the second inequality is not true. In an effort to prove \eqref{mono} we deduced the parts {\bf e} and {\bf f} of Theorem 1, however, these do not help us in the proof of \eqref{mono}. More precisely, by using part {\bf f} of Theorem 1, we clearly have that
$$\Phi(s+t)\geq \sqrt{\frac{s+\frac{1}{4}}{s+t+\frac{1}{4}}}\Phi(s)$$
for all $s,t>0,$ and consequently
$$\frac{\Phi(2r)}{p(2r)}\leq \frac{2r+1}{2r+2}\sqrt{\frac{2r+2+\frac{1}{4}}{2r+\frac{1}{4}}}\cdot\frac{\Phi(2r+2)}{p(2r+2)},$$
however, this is not less or equal than
${\Phi(2r+2)}/{p(2r+2)}.$ Summarizing, the proof of \cite[Lemma 4.4]{kanter} is not complete, which implies that the proof of \cite[Lemma 4.3]{kanter}
is not complete too, and then the proof of \cite[Theorem 4.1]{kanter} is not correct. Fortunately, \cite[Theorem 4.1]{kanter} has been proved recently by Mattner and Roos \cite{roos} by using a different approach, and this in particular implies that the inequality \eqref{conj} is valid for $r\in\{0,1,2,\dots\}.$ More precisely, Kanter \cite[p. 222]{kanter} pointed out that if the convex set is $\{0\}$ and the random variables are considered to be $\pm 1$ valued random variables, then in particular \cite[Theorem 4.1]{kanter} implies \eqref{conj}.

Now, in what follows we present an analytic proof of the extension of \eqref{conj} to $r\geq0.$ Namely, we prove the following result.

\begin{theorem}
The extension of Kanter's inequality \eqref{conj} to real variable, that is,
   \begin{equation} \label{conj2}
      \Phi(2r)=e^{-2r}\left[I_0(2r)+I_1(2r)\right]\geq \frac{\Gamma(2r+1)}{\Gamma^2(r+1)}2^{-2r},
   \end{equation}
is valid for $r\geq0.$
\end{theorem}

\begin{proof}
Observe that to prove \eqref{conj2} it is enough to show that \eqref{conjinteg} is valid for all $r\geq0.$ The case $r=0$ is obvious, and thus in what follows we assume that $r>0.$ First we consider the integral
   \[ \varphi(r) = \dfrac1\pi\, \int_0^\pi (\cos t)^{2r} (1+ \cos t)dt\]
and we show that the right-hand side of \eqref{conjinteg} and \eqref{conj2} coincide. Substituting $x = \cos t$ we get
   \[ \varphi(r) = \dfrac2\pi \int_0^1 \dfrac{x^{2r}}{\sqrt{1-x^2}}dx
           = \dfrac1\pi {\rm B}\left( r + \frac12, \frac12\right)
           = \dfrac{\Gamma\left(r+\frac12\right)}{\sqrt{\pi} \Gamma(r+1)}.\]
Making use of the Legendre duplication formula, $$\Gamma(2z) = \pi^{-\frac{1}{2}}2^{2z-1}\Gamma(z) \Gamma\left(z+\frac12\right)$$ for
$z=r+\frac{1}{2}$, we immediately conclude that
   \[ \varphi(r) = \dfrac{\Gamma(2r+1)}{\Gamma^2(r+1)} {2^{-2r}},\]
as we required.

Now, let us consider the integral
   \[ A(r) = \dfrac1\pi\,\int_0^\pi \left[ {\rm e}^{2r(\cos t -1)}-(\cos t)^{2r}\right](1+\cos t)\,{\rm d}t.\]
To prove \eqref{conjinteg} we show that $A(r)$ is positive for $r>0.$ For this we shall use the Okamura's variant of the second integral mean--value theorem \cite{matsumoto,okamura}, which states that if
$f:[a,b] \to \mathbb R$ is a monotone function and $g:[a,b]\to \mathbb R$ is integrable, then there exists a $\xi \in [a,b]$ such that
   \begin{equation} \label{oka}
      \int_a^b f(t)g(t)dt = f(a^+) \int_a^\xi g(t)dt + f(b^-)\int_\xi^b g(t)dt\,.
   \end{equation}
Choosing $[a,b]=[0, \pi],$ $t\mapsto f(t) = 1+\cos t,$ which is monotone (decreasing) and $t\mapsto g(t) = {\rm e}^{2r(\cos t -1)}-(\cos t)^{2r},$ which is
integrable, for a fixed $0 \le \xi \le \pi$, by \eqref{oka} we have
   \[ A(r) = \frac{2}{\pi}\int_0^\xi \left[ {e}^{2r(\cos t -1)} - (\cos t)^{2r}\right]dt
               = \frac{2}{\pi}\int_{\cos \xi}^1 \left[ {e}^{2r(t -1)} - t^{2r}\right]\dfrac{dt}{\sqrt{1-t^2}}.\]
Let $h_r(t)$ denote the integrand of the last integral, that is,
   \[ h_r(t) = \dfrac{{\rm e}^{2r(t -1)} - t^{2r}}{\sqrt{1-t^2}}.\]
Now, we consider two cases. The first case is when $\xi \in \left[0, \frac{\pi}{2}\right).$ In this case the integrand $h_r(t)$ is positive for all $r>0$,
since ${e}^{t-1} \ge 1+(t-1) = t > 0$ for $t\in [\cos \xi, 1]$. Consequently, we have $A(r)>0$ for $r>0.$ The second case is when $\xi \in \left[\frac{\pi}{2}, \pi\right],$ that is, $-1 \le
\cos \xi \le 0.$ Observe that in this case we have
   \[ A(r) = \dfrac2\pi\, \left[ \int_{\cos \xi}^0h_r(t)dt + \int_0^1h_r(t)dt\right] =
            A_\xi + B,\]
and since $t\mapsto {e}^{2r(t-1)}$ is monotone increasing on $\mathbb R$ for $r>0,$ and $t\mapsto t^{2r}(1-t^2)^{-\frac12}$ is an even function, the
first integral's modulus $|A_\xi|$ cannot overgrow $B$. Indeed,
   \begin{align*}
      \dfrac\pi2\, |A_\xi| &= \int_0^{-\cos \xi} h_r(-t)\,{\rm d}t < \int_0^{-\cos \xi} h_r(t)dt\\
                           &\le \int_0^{-\cos \xi} h_r(t)\, {\rm d}t + \int_{-\cos \xi}^1 h_r(t)dt
                            = \dfrac\pi2\,B\, .
   \end{align*}
Hence $A(r)$ is nonnegative for all $r>0.$ This completes the proof.\end{proof}

\section*{Open Problem} Finally, motivated by the results of section 2 we pose the following problem: find a generalization of Kanter's inequality \eqref{conj2} for $\Phi_{\nu}.$

\end{document}